\documentclass[11pt]{amsart}
\title[The special shadow-complexity of $\#_k(S^1\times S^3)$]
{The special shadow-complexity of $\#_k(S^1\times S^3)$}

\author{Hironobu Naoe}
\address{
Department of Mathematics, Chuo University, 1-13-27 Kasuga Bunkyo-ku, Tokyo, 112-8551, Japan}
\email{naoe@math.chuo-u.ac.jp}

\thanks{The author is supported by JSPS KAKENHI Grant Number JP20K14316.}

\usepackage{amsfonts,amsmath,amssymb,amscd}
\usepackage{amsthm}
\usepackage{latexsym}
\usepackage{bm}
\usepackage[dvipdfmx]{graphicx}
\usepackage{psfrag}
\usepackage{color}
\usepackage{url}
\usepackage{ulem}
\usepackage{pinlabel}

\theoremstyle{plain}
\newtheorem{theorem}{Theorem}[section]

\newtheorem{corollary}[theorem]{Corollary}

\theoremstyle{definition}
\newtheorem{definition}[theorem]{Definition}

\theoremstyle{remark}
\newtheorem{claim}{Claim}

\newtheoremstyle{mycitation}%
    {}%
    {}%
    {\it}%
    {}%
    {\bf}%
    {}%
    {5pt}%
    {\thmname{#1} \thmnumber{#2}\thmnote{#3}.}%
\theoremstyle{mycitation}

\newcommand{\Z}{\mathbb{Z}}

\newcommand{\CP}{\mathbb{CP}^2}
\newcommand{\mCP}{\overline{\mathbb{CP}}^2}

\newcommand{\Int}{\mathrm{Int}}

\newcommand{\Nbd}{\mathrm{Nbd}}

\newcommand{\gl}{\mathfrak{gl}}

\newcommand{\shco}{\mathrm{sc}}
\newcommand{\spshco}{\mathrm{sc}^{\mathrm{sp}}}

\definecolor{darkred}{rgb}{.80,.0,.0}

\definecolor{bl}{gray}{0.7}

\newlength{\myheight}
\newlength{\myheighta}
\makeatletter

\long\def\@makecaption#1#2{
  \small
  \vskip\abovecaptionskip
  \sbox\@tempboxa{#1. #2}
  \ifdim \wd\@tempboxa >\hsize
    #1. #2\par
  \else
    \global \@minipagefalse
    \hb@xt@\hsize{\hfil\box\@tempboxa\hfil}
  \fi
  \vskip\belowcaptionskip}

\makeatother
\begin{document}
\begin{abstract}
The special shadow-complexity is an invariant of closed $4$-manifolds
defined by Costantino using Turaev's shadows. 
We show that for any positive integer $k$, 
the special shadow-complexity of the connected sum of $k$ copies of $S^1\times S^3$ is exactly $k+1$. 
\end{abstract}

\maketitle
\section{Introduction}
A {\it shadow} was introduced by Turaev in 1990s for the study of quantum invariants of 
$3$-manifolds and links~\cite{Tur94}. 
It is defined as a certain $2$-dimensional polyhedron 
embedded in a smooth $4$-manifold. 
Turaev showed that any (closed) smooth $4$-manifold can be described 
by a shadow together with a coloring on regions by half-integers, 
which is called the {\it gleam}. 
There are various studies concerning shadows, 
see~\cite{CM17,Cos06,Cos06b,Cos08,CT08,IK14,KMN18,KN20,Mar11} for instance. 

A combinatorial description such as a shadow provides a framework 
according to the ``complexity'' of the description. 
Costantino applied shadows to define invariants of $4$-manifolds, 
called {\it the shadow-complexity} and {\it the special shadow-complexity}, 
with inspired by the Matveev complexity of $3$-manifolds. 
The (special) shadow-complexity of a (closed) $4$-manifold $W$ 
is defined as the minimum of the number of {\it true vertices} of all (special) shadows of $W$. 
Roughly speaking, 
it measures how complicated shadows of a given $4$-manifold need to be. 

In general, it is difficult to determine the value of an invariant defined as the minimum of something, 
especially to give a lower bound for it, 
and the (special) shadow-complexity is no exception. 
We refer the reader to~\cite{CT08} 
for the only known result giving a lower bound for the 
shadow-complexity of $4$-manifolds with non-empty boundary. 
In that paper, Costantino and Turston have established a relationship between 
the shadows and the geometric structures of $3$-manifolds. 
On the other hand, the shadow-complexity of closed $4$-manifolds has not 
been bounded from below by anything. 
Note that the unboundedness of the shadow-complexity of closed $4$-manifolds 
was shown in~\cite{KMN18}, but their proof does not provide a concrete lower bound. 

The following is our main theorem, 
which determines the explicit values of the special shadow-complexities 
for infinitely many closed $4$-manifolds. 
\begin{theorem}
\label{thm:mainthm}
For any positive integer $k$, $\spshco(\#_k(S^1\times S^3))=k+1$. 
\end{theorem}
An upper bound of $\spshco(\#k(S^1\times S^3))$ can be constructed by an easy observation. 
Actually, $\spshco(S^1\times S^3)=2$ is easily checked 
by constructing a concrete shadow 
and the classification of the closed $4$-manifolds with special shadow-complexity at most $1$ in~\cite{Cos06b}. 
As Costantino mentioned in~\cite{Cos06b}, 
the inequality $\spshco(W\# W')\leq\spshco(W)+\spshco(W')+4$ holds 
for any closed $4$-manifolds $W$ and $W'$. 
Therefore, we obtain $\spshco(\#_k(S^1\times S^3))\leq 6k-4$, but it is too rough. 

In order to give the upper bound $\spshco(\#_k(S^1\times S^3))\leq k+1$, 
we construct a shadow of $\#_k(S^1\times S^3)$ concretely. 
For this construction, we will introduce two techniques, 
boundary-disposal and vertex-creation. 
On the other hand, in order to give the lower bound 
$\spshco(\#_k(S^1\times S^3))\geq k+1$, 
we investigate presentations of the fundamental group of shadows of $\#_k(S^1\times S^3)$.

As immediate consequences of the proof of Theorem~\ref{thm:mainthm}, 
we have the following. 
\begin{corollary}[Corollary~\ref{cor:surj}]
The special shadow-complexity of closed $4$-manifolds is a surjection to $\Z_{\geq0}\setminus\{1\}$.
\end{corollary}

\begin{corollary}[Corollary~\ref{cor:lower_bound}]
For any closed $4$-manifold $W$, 
$\mathrm{rank}(\pi_1(W))\leq\spshco(W)$. 
Moreover, if $\pi_1(W)$ is the free group, then $\mathrm{rank}(\pi_1(W))+1\leq\spshco(W)$. 
\end{corollary}

In Section~\ref{sec:Shadows}, 
we review the notion of shadows of smooth $4$-manifold and the (special) shadow-complexity of $4$-manifolds. 
In Section~\ref{sec:Results}, we first introduce two modifications, 
Boundary-disposal and vertex-creation, of shadowed polyhedra 
to construct a special shadow $Z_k$ of $\#_k(S^1\times S^3)$, 
and then we show that $Z_k$ attains the special shadow-complexity of $\#_k(S^1\times S^3)$. 


Throughout this paper, any manifold is supposed to be compact, connected, oriented and smooth 
unless otherwise mentioned. 
\section{Shadows and complexity of $4$-manifolds}
\label{sec:Shadows}
\subsection{Simple polyhedra and shadowed polyhedra}
We first introduce some terminology. 
A {\it simple polyhedron} is a connected compact space $X$ 
such that a regular neighborhood $\Nbd(x;X)$ of each point $x\in X$ 
is homeomorphic to one of (i)-(iv) shown in Figure~\ref{fig:local_model}. 
\begin{figure}[tbp]
\labellist
\footnotesize\hair 2pt
\pinlabel (i) [B] at    48.19 8.73
\pinlabel (ii) [B] at  147.40 8.73
\pinlabel (iii) [B] at 246.61 8.73
\pinlabel (iv) [B] at  334.49 8.73
\endlabellist
\centering
\includegraphics[width=.65\hsize]{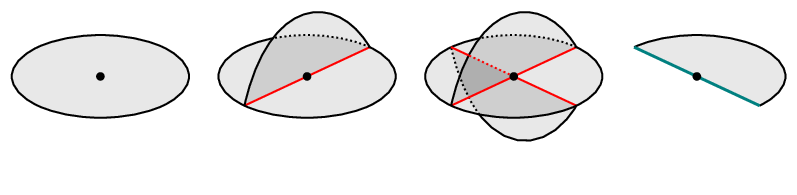}
\caption{Local models of simple polyhedra.}
\label{fig:local_model}
\end{figure}
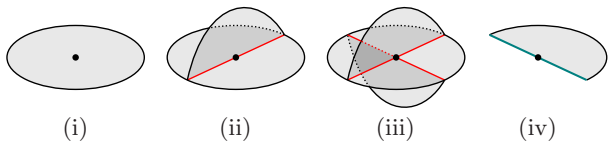
A point of type~(iii) is called a {\it true vertex}. 
The set of all points of types (ii) and (iii) is called the {\it singular set} of $X$, 
which is denoted by $S(X)$. 
Note that each connected component of $S(X)$ is a circle or a quartic graph. 
A connected component of $S(X)$ with the true vertices removed is called a {\it triple line}. 
Each connected component of $X\setminus S(X)$ is called a {\it region}. 
If every region of $X$ is an open disk, then $X$ is said to be {\it special} 
and is called a {\it special polyhedron}. 
The set of points of type~(iv) is the {\it boundary} of $X$, 
which is denoted by $\partial X$. 
If $\partial X$ is empty, the simple polyhedron $X$ is said to be {\it closed}. 
Thus, a special polyhedron is closed. 
A region not intersecting $\partial X$ is called an {\it internal region}. 
A {\it boundary region} is a region that is not an internal region. 

We then define the $\Z_2$-{\it gleam} of a simple polyhedron $X$. 
Let $R$ be an internal region of $X$. 
Then $R$ is homeomorphic to the interior of some compact surface $F$, 
and the homeomorphism $\Int F\to R$ will be denoted by $f$. 
This $f$ can extend to a local homeomorphism $\overline{f}:F\to X$. 
Moreover, there exists a simple polyhedron $\widetilde{F}$ obtained from $F$ 
by attaching an annulus or a M\"obius band to each boundary component of $F$ along the core circle 
such that $\overline{f}$ can extend to a local homeomorphism $\widetilde{f}:\widetilde{F}\to X$. 
Then the number of the M\"obius bands attached to $F$ modulo $2$ 
is called the $\Z_2$-{\it gleam} of $R$ and is denoted by $\gl_2(R)\in\{0,1\}$. 
Note that for each region, 
its $\Z_2$-gleam is determined only by the combinatorial structure of $X$. 

A function mapping each internal region $R$ of $X$ 
to a half-integer $\gl(R)$ satisfying $\gl(R)+\frac12\gl_2(R)\in\Z$
is called a {\it gleam function}, or simply {\it gleam}, of $X$. 
We also call the value $\gl(R)$ the {\it gleam} of $R$. 
A simple polyhedron equipped with a gleam is called a {\it shadowed polyhedron}. 
\subsection{Shadows of $4$-manifolds}
Let $X$ be a simple polyhedron. 
A $4$-manifold $M$ with boundary is called a {\it $4$-dimensional thickening} of $X$ if 
the following hold;
\begin{itemize}
 \item
$X$ is embedded in $M$ local-flatly, that is, 
a regular neighborhood $\Nbd(x;X)$ of each point $x\in X$ is contained in a smooth $3$-ball in $M$, 
 \item 
$M$ collapses onto $X$ under some triangulations agreeing with the smooth structure of $M$, and 
 \item 
$X\cap\partial M=\partial X$. 
\end{itemize} 
We call $X$ a {\it shadow} of $M$, 
which is the definition of shadows of $4$-manifolds with boundary. 
Shadows of closed $4$-manifolds will be defined later. 

Any simple polyhedron has a $4$-dimensional thickening. 
Although there might be non-diffeomorphic $4$-dimensional thickenings for a single simple polyhedron, 
Turaev showed that there exists a canonical way to associate to each shadowed polyhedron $X$ 
a unique $4$-dimensional thickening $M_X$ up to diffeomorphism. 
This correspondence is called {\it Turaev's reconstruction}. 

We give a brief review of Turaev's reconstruction. 
Let $X$ be a shadowed polyhedron equipped with a gleam $\gl$. 
Let $R_1,\ldots,R_m$ denote the internal regions of $X$. 
Set $\bar R_i=R_i\setminus \Int \Nbd(S(X);X)$ for $i\in\{1,\ldots,m\}$ and 
$X_S=X\setminus \Int (\bar R_1\sqcup\cdots\sqcup \bar R_m)$. 
Let $\varphi:\partial \bar R_1\sqcup\cdots\sqcup\partial \bar R_m\to \partial X_S\setminus\partial X$ 
denote a homeomorphism reconstructing $X$ as $(\bar R_1\sqcup\cdots\sqcup \bar R_m\sqcup X_S)/\varphi$. 
It is easy to see that there exists a $3$-dimensional (possibly non-orientable) handlebody $N_S$ 
in which $X_S$ is properly embedded such that $N_S$ collapses onto $X_S$. 
Set $B_S=\Nbd(\partial X_S;\partial M_S)$, which consists of disjoint annuli or M\"obius bands. 
Let $M_S$ be the subbundle of the determinant line bundle over $N_B$ 
with fibers of length $\leq1$ with respect to an auxiliary Riemannian metric. 
Note that $M_S$ is a $4$-dimensional thickening of $X_S$ 
with $B_N$ embedded in the boundary $\partial M_S$. 
Set $N_i=\bar R_i\times [-1,1]$, and suppose $\bar R_i$ is embedded in $N_i$ 
by identifying $\bar R_i$ with $\bar R_i\times \{0\}$. 
Note that $N_i$ is a non-orientable $3$-manifold if $\bar R_i$ is non-orientable. 
Set $A_i=\partial \bar R_i\times [-1,1]\subset N_i$. 
Let $M_i$ be the subbundle of the determinant line bundle over $N_i$ 
with fibers of length $\leq1$ with respect to an auxiliary Riemannian metric, 
so that it contains $\bar R_i$ properly and $A_i$ in the boundary $\partial M_i$. 
For each $i\in\{1,\ldots,m\}$, we glue $M_i$ to $M_S$ by an embedding 
$\Phi_i:\Nbd(\partial \bar R_i;\partial M_i)\to \partial M_S$ 
such that $\Phi_i|_{\partial \bar R_i}=\varphi|_{\partial \bar R_i}$ and 
$\Phi_i(A_i)$ is rotated $\gl(R_i)$ times with respect to $B_i$, 
and the obtained $4$-manifold is nothing but the $4$-dimensional thickening $M_X$ that we wanted. 

We next define shadows of closed $4$-manifolds. 
\begin{definition}
Let $W$ be a closed $4$-manifold and $X$ a simple polyhedron embedded in $W$. 
We call $X$ a {\it shadow} of $W$ if the following hold;
\begin{itemize}
 \item
$X$ is locally-flat in $W$, and
 \item 
$W\setminus\Int\Nbd(X;W)$ is diffeomorphic to $\natural_k(S^1\times B^3)$ for some $k\in\Z_{\geq0}$. 
\end{itemize}
\end{definition}

By the second condition in the definition, 
$X$ can be regraded as a $2$-skeleton of $W$ since 
$W$ is obtained from $\Nbd(X;W)$ by attaching some $3$-handles and $4$-handles. 
Therefore, we have $H_1(X)\cong H_1(W)$ and $\pi_1(X)\cong \pi_1(W)$. 

The notion of shadows was introduced by Turaev in order to study quantum invariants of $3$-manifolds and links, and he showed the following. 
\begin{theorem}[Turaev~\cite{Tur94}]
\label{thm:Turaev}
Any closed $4$-manifold admits a shadow. 
\end{theorem}

Suppose a shadow $X$ of a closed $4$-manifold $W$ is given. 
Turaev showed that there exists a gleam of $X$ 
such that the $4$-dimensional thickening $M_X$ of the shadowed polyhedron $X$ 
is diffeomorphic to $\Nbd(X;W)$, see~\cite{Cos05,Tur94}. 

By the definition of shadows, $\partial \Nbd(X;W)$ is diffeomorphic to $\#_k(S^1\times S^2)$ for some $k\in\Z_{\geq0}$. 
By Laudenbach and Po\'enaru~\cite{LP72}, 
a closed $4$-manifold obtained from $\Nbd(X;W)$ by attaching $3$- and $4$-handles is unique up to diffeomorphism, 
which implies that any closed $4$-manifold can be described by a shadowed polyhedron. 

\subsection{Complexity of $4$-manifolds}
The {\it complexity} of a simple polyhedron $X$ is defined as the number $c(X)$ of true vertices of $X$. 
Costantino defined two kinds of invariants of $4$-manifolds using shadows in~\cite{Cos06b}. 
\begin{definition}
Let $W$ be a $4$-manifold with possibly non-empty boundary. 
\begin{enumerate}
 \item
The {\it shadow-complexity}, denoted by $\shco(W)$, of $W$ is the minimum of $c(X)$ for all shadows $X$ of $W$. 
 \item 
The {\it special shadow-complexity}, denoted by $\spshco(W)$, of $W$ is the minimum of $c(X)$ for all special shadows $X$ of $W$. 
\end{enumerate}
\end{definition}
There exist exotic (i.e. homeomorphic but not diffeomorphic) smooth structures in dimension $4$, 
and it is known that the shadow-complexity and the special shadow-complexity 
are invariants essentially for the smooth structures of $4$-manifolds~\cite{Mar11}. 

Our main object is the special shadow-complexity,
which has a remarkable property: 
\begin{theorem}
[\cite{Cos06b}, see also~\cite{Mar05}]
The special shadow-complexity of closed $4$-manifolds is a finite-to-one invariant.  
\end{theorem}
Costantino classified all the closed $4$-manifolds with special shadow-complexity at most $1$. 
\begin{theorem}
[{\cite[Theorem 1.1]{Cos06b}}]
\label{thm:sc^sp_at_most_one}
Let $W$ be a cloed $4$-manifold. 
The following are equivalent;
\begin{itemize}
 \item
$\spshco(W)=0$, 
 \item
$\spshco(W)\leq 1$, and 
 \item
$W$ is diffeomorphic to either $S^4$, $\CP$, $\mCP$, $\CP\#\CP$, $\mCP\#\mCP$, $\CP\#\mCP$ or 
$S^2\times S^2$. 
\end{itemize}
\end{theorem}
On the other hand, while the shadow-complexity is not finite-to-one, 
Martelli characterized all the closed $4$-manifolds with shadow-complexity $0$ in \cite{Mar11}. 
We also refer the reader to~\cite{KMN18} for the characterization of 
closed $4$-manifolds with {\it connected shadow-complexity} at most $1$, 
where the connected shadow-complexity is another kind of invariant defined by using shadows. 
See~\cite{KMN18} for more details. 

There have been no results to answer a question: 
for an arbitrarily given $n\in\Z_{\geq0}$, 
does there exist a closed $4$-manifold whose (ordinary, special or connected) 
shadow-complexity is exactly $n$? 
Note that we can easily give an example of a $4$-manifold with boundary 
whose (special) shadow-complexity is exactly $n$, see~\cite{IK14}. 
\section{Results}
\label{sec:Results}
For a positive integer $k$, 
let $Z_k$ be a special polyhedron described in the lowermost part of Figure~\ref{fig:Xk}, 
which will be explained in more detail in 
Subsection~\ref{subsec:The_special_polyhedra_$X_k$_and_$Z_k$}. 
Our theorem determines the special shadow-complexity of $\#_k(S^1\times S^3)$, 
and actually $Z_k$ is a shadow of $\#_k(S^1\times S^3)$ attaining the special shadow-complexity. 
We start this section with explaining how we find $Z_k$. 

\subsection{Boundary-disposal and vertex-creation}
Here we introduce two kinds of modifications used to construct $Z_k$. 

We first explain a modification called a {\it boundary-disposal}. 
Let $X$ be a shadowed polyhedron with 
$S(X)\ne\emptyset$ and 
$\partial X$ consisting of one circle $C_0$, 
so $X$ has a single boundary region, say $R_0$. 
Let us consider an arc $\gamma$ contained in $R_0$ 
such that it connects a point in $S(X)$ and a point in $C_0$. 
See the upper left part of Figure~\ref{fig:specialization}, 
which depicts $\Nbd(\gamma\cup C_0;X)$. 
\begin{figure}[tbp]
\labellist
\footnotesize\hair 2pt
\pinlabel {\textcolor[rgb]{0, 0.50196, 0.50196}{$C_0$}} [Bl] at 132.23 217.66
\pinlabel $1/2$ [B] at 313.39 227.58
\pinlabel $1/2$ [B] at 274.96 227.58
\pinlabel  -$1$ [B] at 232.44 209.99
\pinlabel  glue [B] at  59.53 164.64
\pinlabel  glue [B] at 272.13 164.64
\pinlabel   (i) [B] at 189.92 238.34
\pinlabel  (ii) [B] at 189.92 153.30
\pinlabel (iii) [B] at 189.92  82.43
\pinlabel $1/2$ [B] at 107.72  48.42
\pinlabel   $0$ [B] at  96.38  75.51
\pinlabel $1/2$ [B] at  62.36  69.68
\pinlabel  -$1$ [B] at  19.84  54.09
\pinlabel  glue [B] at  59.53   8.73
\pinlabel  glue [B] at 272.13   8.73
\pinlabel   $0$ [B] at 308.98  75.51
\pinlabel $1/2$ [B] at 274.96 69.68
\pinlabel -$1/2$ [B] at 320.31 48.42
\pinlabel {\textcolor[rgb]{0, 0.50196, 0}{$\gamma$}} [B] at 100.80 234.67
\endlabellist
\centering
\includegraphics[width=1\hsize]{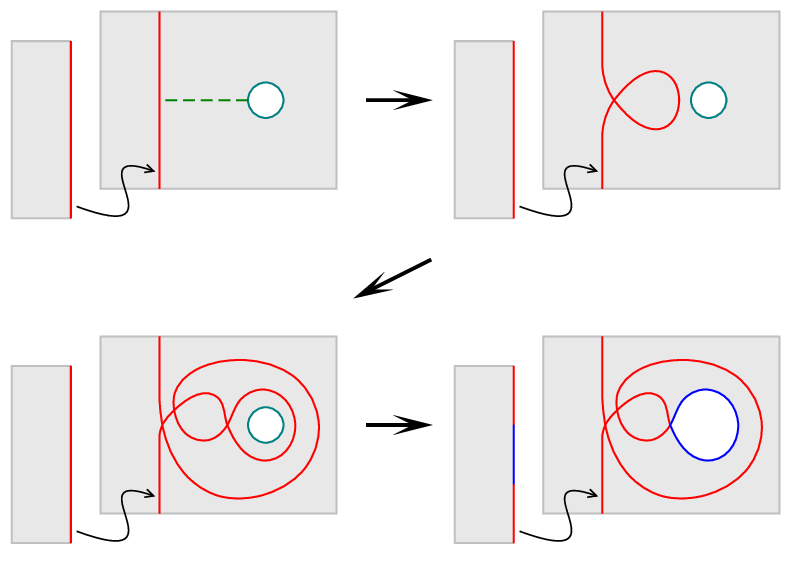}
\caption{Bonundary-disposal.}
\label{fig:specialization}
\end{figure}
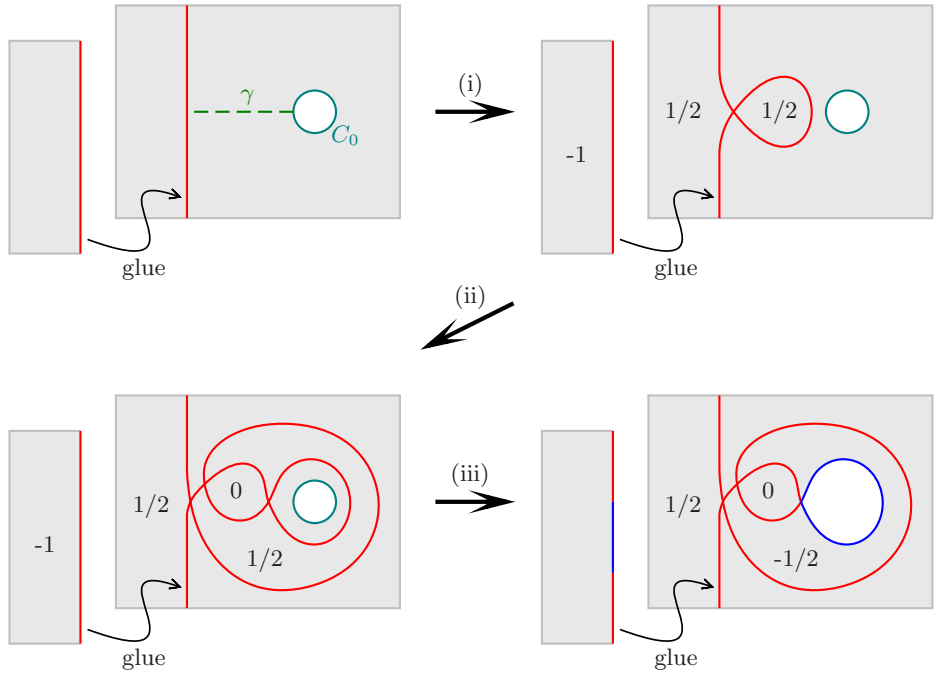
\begin{figure}[tbp]
\labellist
\footnotesize\hair 2pt
\pinlabel  glue [B] at  59.53   8.73
\pinlabel  glue [B] at 215.43   8.73
\pinlabel  \textcolor{red}{$\delta$} [Bl] at 79.37 70.94
\endlabellist
\centering
\includegraphics[width=.8\hsize]{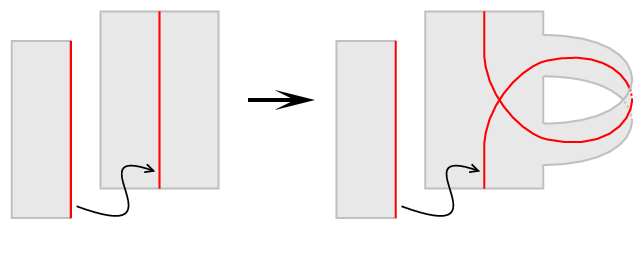}
\caption{Vertex-creation.}
\label{fig:creation}
\end{figure}
The moves shown in Figures~\ref{fig:specialization}-(i) and -(ii) modify $X$ 
into another shadowed polyhedron without 
changing the corresponding $4$-manifold~\cite{Tur94}. 
By the moves (i) and (ii), exactly $3$ true vertices are created in total. 
The move shown in Figures~\ref{fig:specialization}-(iii) is a collapsing 
so that the annular boundary region is removed, 
which also removes one true vertex. 
As a result, 
the shadowed polyhedron $X$ is modified into another one $X'$ such that
\begin{itemize}
 \item
the corresponding $4$-manifold of $X'$ is the same as $X$, 
\item
exactly one disk region is newly created, 
\item
the homeomorphism types of the original regions are not changed except for $R_0$, 
and $R_0$ is changed into a surface homeomorphic to $R_0\cup_{C_0} D^2$, 
especially $\partial X'=\emptyset$, and 
\item
$c(X')=c(X)+2$. 
\end{itemize}
The composition of the moves (i), (ii) and (iii) is called a {\it boundary-disposal}. 
We stress that the resulting polyhedron is special 
if $X\cup_{\partial X}D^2$ is special. 
By this modification, 
the singular set is changed as shown in the upper part of Figure~\ref{fig:moves}. 
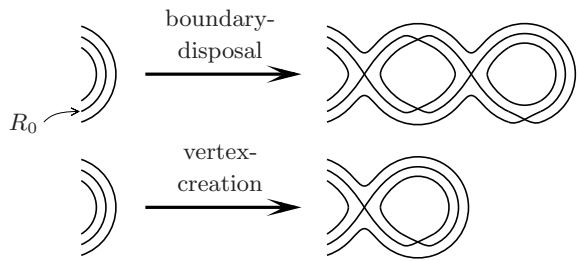
\begin{figure}[tbp]
\labellist
\footnotesize\hair 2pt
\pinlabel $R_0$           [Br] at  15.12  73.66
\pinlabel {vertex-}     [B]  at 108.67  58.23
\pinlabel {creation}    [B]  at 108.67  41.23
\pinlabel {boundary-}   [B]  at 108.67 129.10
\pinlabel {disposal} [B]  at 108.67 112.10
\endlabellist
\centering
\includegraphics[width=.6\hsize]{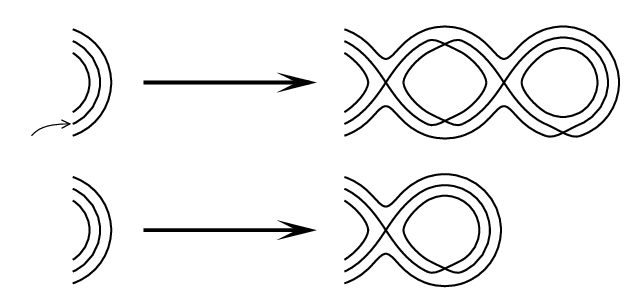}
\caption{Changes on the singular sets by the modifications boundary-disposal and vertex-creation, 
where we only illustrate how the regions are attached along the singular sets. 
The singular sets themselves are not pictured in the figure. }
\label{fig:moves}
\end{figure}

We next introduce another modification called {\it vertex-creation}. 
Let $X$ be a special polyhedron. 
Then, the homeomorphism type of $X$ is 
determined only by a regular neighborhood $\Nbd(S(X);X)$ of its singular set. 
Let $\delta$ be a short arc contained in a triple line of $X$, 
see the left part of Figure~\ref{fig:creation}. 
Then we consider another special polyhedron $X'$ such that 
$\Nbd(S(X');X')$ is obtained from $\Nbd(S(X);X)$ by modifying 
near $\delta$ as shown in the figure. 
Although the simple polyhedron $X$ and the resulting one $X'$ 
do not necessarily share the same $4$-dimensional thickening, 
we emphasize the following properties; 
\begin{itemize}
 \item
the numbers of regions of $X$ and $X'$ are equal, and 
\item
$c(X')=c(X)+1$, especially $S(X')$ is homotopy equivalent to the wedge sum of $S(X)$ and a circle. 
\end{itemize}
The modification to get $X'$ from $X$ is called a {\it vertex-creation}, 
by which the singular set is changed as shown in the lower part of Figure~\ref{fig:moves}. 

\subsection{The special polyhedra $X_k$ and $Z_k$}
\label{subsec:The_special_polyhedra_$X_k$_and_$Z_k$}
First, let $X_1$ be a special polyhedron whose singular set is indicated in 
the uppermost part of Figure~\ref{fig:Xk}. 
It has a single disk region, and $S(X)$ is homeomorphic to a circle. 
We next define $X_2$ as a special polyhedron obtained from $X_1$ by a vertex-creation, 
see the second from the top in Figure~\ref{fig:Xk}.  
Then $X_2$ also has a single disk region, and $S(X)$ is 
homeomorphic to the bouquet of two circles. 
For $k\geq3$, let $X_k$ be a special polyhedron obtained from $X_{k-1}$ by a vertex-creation, 
see Figure~\ref{fig:Xk}. 
The special polyhedron $X_k$ has a single disk region, 
and $S(X_k)$ is homotopy equivalent to the bouquet of $k$ circles. 
Note that $c(X_k)=k-1$. 
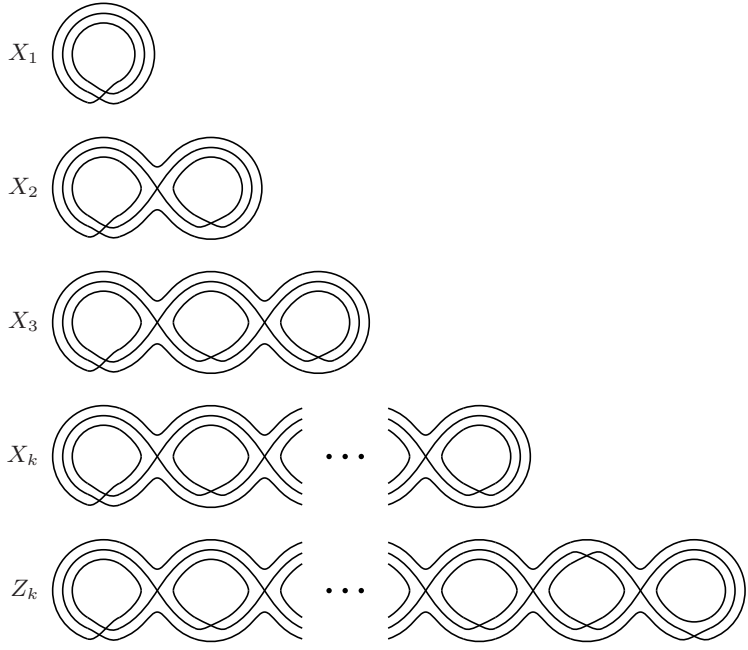
\begin{figure}[tbp]
\labellist
\footnotesize\hair 2pt
\pinlabel   $X_3$ [Br] at 25.54 171.46
\pinlabel   $X_2$ [Br] at 25.54 242.32
\pinlabel   $X_1$ [Br] at 25.54 313.19
\pinlabel   $X_k$ [Br] at 25.54 100.59
\pinlabel   $Z_k$ [Br] at 25.54 29.7
\endlabellist
\centering
\includegraphics[width=.8\hsize]{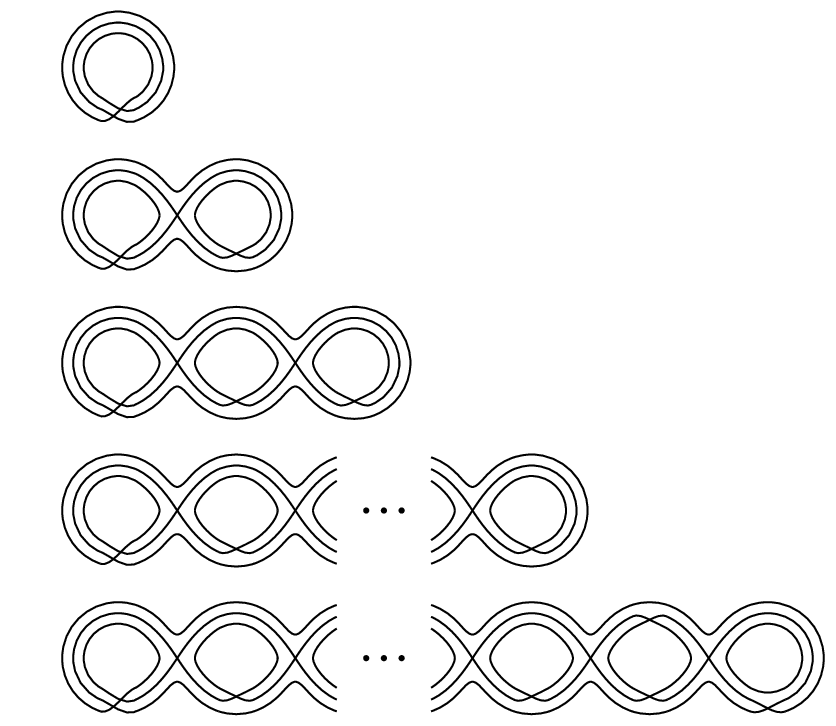}
\caption{The special polyhedra $X_1,X_2,X_3,X_k$ and $Z_k$, 
where we only illustrate how the regions are attached along the singular sets. 
The singular sets themselves are not pictured in the figure. }
\label{fig:Xk}
\end{figure}

For each $k\in\Z_{\geq1}$, 
let $X_k^\circ$ be a simple polyhedron obtained from $X_k$ 
by removing an open disk from the single region of $X_k$. 
It is easy to see that $X_k^\circ$ has a single annular boundary region. 
Then $X_k^\circ$ collapses onto $S(X_k^\circ)$, which is homotopy equivalent to 
the bouquet of $k$ circles. 
Hence $X_k^\circ$ has a unique $4$-dimensinoal thickening $\#_k(S^1\times B^3)$, 
and it can be regarded as a shadow of $\#_k(S^1\times S^3)$. 

We define $Z_k$ as a simple polyhedron obtained from $X_k^\circ$ by a boundary-disposal. 
By the construction, $Z_k$ is a special shadow of $\#_k(S^1\times S^3)$. 
Moreover, we have $c(Z_k)=k+1$, and $Z_k$ has exactly two disk regions. 

\subsection{The proof of Theorem~\ref{thm:mainthm}}
We finally give the proof of Theorem~\ref{thm:mainthm}. 
\begin{proof}[Proof of Theorem~\ref{thm:mainthm}]
We already have $\shco(\#_k(S^1\times S^3))\leq k+1$ due to the existence of $Z_k$, 
and then we show $\shco(\#_k(S^1\times S^3))\geq k+1$ below.

Let $X$ be a special shadow of $\#_k(S^1\times S^3)$ with $c(X)=n$, 
so it suffices to show that $n\geq k+1$. 
By the classification of closed $4$-manifolds with 
special shadow-complexity zero~\cite[Theorem 1.1]{Cos06b}, 
we need $n\geq 2$. 
Let $m$ be the number of regions of $X$, which is required to satisfy $m\geq2$ by~\cite[Corollary 3.17]{Cos06b}. 

Since $X$ has at least $2$ regions, 
there does not exist a triple line 
such that only a single region passes through it. 
Therefore, 
there exists a triple line $a_1$ and a region $R_1$ passing through $a_1$ exactly once. 
Take a spanning tree $T$ of $S(X)$ so that $T$ does not contains $a_1$, 
which can be done since $S(X)$ is a quartic graph. 
It is easily seen that $S(X)\setminus (T\cup a_1)$ consists of $n$ triple lines, 
which will be denoted by $a_2,\ldots,a_{n+1}$. 
The triple lines in $T$ are denoted by $a_{n+2},\ldots,a_{2n}$. 
Let $R_2,\ldots,R_m$ denote the regions other than $R_1$. 
We give orientations to triple lines and regions arbitrarily. 
For each $j\in\{1,\ldots,m\}$, 
we will consider the region $R_j$ as the image of the interior of a closed disk $D_j$ attached along $S(X)$ by some attaching map $\varphi_j$. 
Then the preimage of true vertices by $\varphi_j$ devides $\partial D_j$ into the preimage of some triple lines, 
which gives a word $\tilde r_j$ in $\{a_1,\ldots,a_{2n}\}$. 
Note that we will not distinguish two words forming $a_i^{\pm1}w$ and $wa_i^{\pm1}$ 
for some $i\in\{1,\ldots,2n\}$ and some word $w$. 
We then remove alphabets $a_{n+2},\ldots,a_{2n}$ and their inverses from $\tilde r_j$, 
so that we obtain a word in  $\{a_1,\ldots,a_{n+1}\}$. 
Let $r_j$ denote the obtained word. 
The above removal corresponds to the quotient $X\to X/T$, 
and by the homotopy equivalence $X\simeq X/T$, 
$\pi_1(X)$ admits a presentation 
\[
\pi_1(X)\cong 
\langle
a_1,\ldots,a_{n+1}\mid r_1,\ldots,r_m
\rangle.
\]

We need the following three claims. 
\begin{claim}
\label{clm:triple_line}
For each $i\in\{1,\ldots,n+1\}$, the total number of times $a_i$ and $a_i^{-1}$ appear in $r_1,\ldots,r_m$ is exactly $3$. 
\end{claim}
\begin{proof}
It follows from that 
for each triple line, the number of regions (counted with multiplicity) 
pass through the triple line is 3. 
\end{proof}
\begin{claim}
\label{clm:all_reduced}
Each word of $r_1,\ldots,r_m$ is reduced, 
that is, it does not contain a subword forming $a_ia_i^{-1}$ or $a_i^{-1}a_i$ for any $i\in\{1,\ldots,n+1\}$. 
Especially, $r_1,\ldots,r_m$ are not the empty word. 
\end{claim}
\begin{proof}
[Proof of Claim~\ref{clm:all_reduced}]
Let us fix a true vertex $x$. 
Let $a_{i_1}$, $a_{i_2}$, $a_{i_3}$ and $a_{i_4}$ be the four triple lines connecting to the true vertex $x$, 
and suppose their orientations are given so that 
$a_{i_1}^{\epsilon_1}$, $a_{i_2}^{\epsilon_2}$, $a_{i_3}^{\epsilon_3}$ and $a_{i_4}^{\epsilon_4}$ are directed to $x$, 
for some $\epsilon_1,\epsilon_2,\epsilon_3,\epsilon_4\in\{-1,1\}$. 
Note that there may exist $i,i'\in\{i_1,i_2,i_3,i_4\}$ with $i\ne i'$ such that $a_{i}^{\epsilon_i}=a_{i'}^{-\epsilon_{i'}}$, 
see Figure~\ref{fig:claim}. 
\begin{figure}[tbp]
\labellist
\small\hair 2pt
\pinlabel $x$ [Bl] at  105.01 36.43
\pinlabel $a_{i_1}^{\epsilon_1}$ [Bl] at  56.41 52.44
\pinlabel $a_{i_2}^{\epsilon_2}$ [Bl] at  56.41 12.76
\pinlabel $a_{i_3}^{\epsilon_3}$ [Bl] at  95.51 12.76
\pinlabel $a_{i_4}^{\epsilon_4}$ [Bl] at  95.51 49.61
\endlabellist
\centering
\includegraphics[width=.4\hsize]{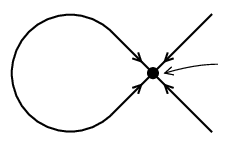}
\caption{An example showing a part of the singular set near a true vertex $x$ 
such that $a_{1}^{\epsilon_1}$ coincides with $a_{2}^{-\epsilon_{2}}$. }
\label{fig:claim}
\end{figure}
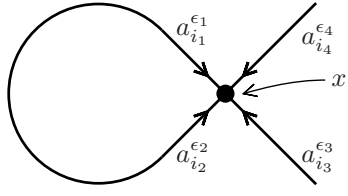
Then a region passing through $x$ is attached along either one of 
$a_{i_1}^{\epsilon_1}a_{i_2}^{-\epsilon_2}$, $a_{i_2}^{\epsilon_2}a_{i_3}^{-\epsilon_3}$, $a_{i_3}^{\epsilon_3}a_{i_4}^{-\epsilon_4}$, $a_{i_4}^{\epsilon_4}a_{i_1}^{-\epsilon_1}$, $a_{i_1}^{\epsilon_1}a_{i_3}^{-\epsilon_3}$ or $a_{i_2}^{\epsilon_2}a_{i_4}^{-\epsilon_4}$. 
Therefore, $\tilde r_1,\ldots,\tilde r_m$ are reduced. 
As already mentioned above, the removal of $a_{n+2}^{\pm1},\ldots,a_{2n}^{\pm1}$ corresponds to the quotient $X\to X/T$. 
Since $T$ is a tree graph, the words $r_1,\ldots,r_m$ are also reduced. 
\end{proof}
\begin{claim}
\label{clm:no_common_subwords}
For $j,j'\in\{1,\ldots,m\}$ with $j\ne j'$ and any subword $w$ of $r_j$ of length at least $2$, 
both $w$ and $w^{-1}$ do not contained in $r_{j'}$. 
\end{claim}
\begin{proof}
[Proof of Claim~\ref{clm:no_common_subwords}]
As seen in the proof of Claim~\ref{clm:all_reduced}, 
each of six regions around a true vertex determines a pair chosen in the four triple line connected to the triple line, 
and this correspondence is one-to-one. 
Hence, $\tilde r_1,\ldots,\tilde r_{m}$ do not contain common subwords. 
Since for any two leaves of $T$, there exists a unique path in $T$ connecting them, 
the words $r_1,\ldots,r_{m}$ also contain no common subwords. 
\end{proof}

By the assumption on $a_1$ and $R_1$ and Claim~\ref{clm:all_reduced}, 
there exists a reduced word $w_1$ not containing $a_1$ or $a_1^{-1}$
such that $r_1=a_1w_1^{-1}$, 
so we have 
\[
\pi_1(X)\cong 
\langle
a_1,\ldots,a_{n+1}\mid a_1w_1^{-1},\,r_2\ldots,r_m
\rangle.
\]
Note that $w_1$ is possibly the empty word. 
We simultaneously remove a generator $a_1$ and a relator $a_1w_1^{-1}$ from the presentation 
with changing other relations suitably, 
and then we obtain a new presentation of $\pi_1(X)$ with $n$ generators. 
Since $\pi_1(X)$ is the free group $F_k$ of rank $k$, 
we have $n\geq k$. 
We want to show that $n\geq k+1$, 
so we suppose that $n=k$ for a contradiction. 

By Claim~\ref{clm:triple_line} and by exchanging the orientations of regions if necessary, 
we can assume without loss of generality that either one of the following holds;
\begin{itemize}
 \item[(i)]
$m\geq3$, and each $r_2$ and $r_3$ contains $a_1$ exactly once, 
 \item[(ii)]
$r_2$ contains $a_1$ and $a_1^{-1}$ exactly once each, or
 \item[(iii)]
$r_2$ contains $a_1$ exactly twice.
\end{itemize}

Assume (i) holds. 
Then we can write $r_2=a_1w_2$ and $r_3=a_1w_3$ for some words $w_2$ and $w_3$ in $\{a_2,\ldots,a_{k+1}\}$, and $\pi_1(X)$ admits a presentation
\[
\pi_1(X)\cong 
\langle
a_2,\ldots,a_{k+1}\mid w_1 w_2 ,\,w_1 w_3,\, r_4,\ldots,r_m
\rangle.
\]
Then, there exists a surjection from the group presented by
\[
\langle
a_2,\ldots,a_{k+1}\mid w_1 w_2
\rangle 
\]
to $\pi(X)$. 
Therefore, $w_1w_2$ must be reduced to the empty word since $\pi_1(X)\cong F_k$. 
If $w_1$ is the empty word, then $w_2$ is reduced to the empty word. 
Then $r_2$ is also reduced to the empty word, which contradicts to Claim~\ref{clm:all_reduced}. 
Suppose $w_1$ is not the empty word. 
If $w_1=a_i$ for some $i\in\{2,\ldots,k+1\}$, then $w_2=a_i^{-1}$. 
Hence $r_2=a_1a_i^{-1}=r_1$, which contradicts Claim~\ref{clm:no_common_subwords}. 
If $w_1$ is of length at least $2$, 
we can check that $w_1$ and $w_2^{-1}$ have a common subword. 
Then $r_1$ and $r_2^{-1}$ also have a common subword, a contradiction to Claim~\ref{clm:no_common_subwords}. 

Assume (ii) holds, 
and suppose that $r_2=x_1 w_2 x_1^{-1} w_2'$ for some words $w_2, w_2'$ in $\{a_2,\ldots,a_{k+1}\}$. 
Then we have 
\[
\pi_1(X)\cong 
\langle
a_2,\ldots,a_{k+1}\mid w_1 w_2 w_1^{-1} w_2',\,r_3,\ldots,r_m
\rangle,
\]
to which there exists the surjection from the group presented by
\[
\langle
a_2,\ldots,a_{k+1}\mid w_1 w_2 w_1^{-1} w_2'
\rangle. 
\]
Thus, $w_1w_2w_1^{-1}w'_2$ must be reduced to the empty word since $\pi_1(X)\cong F_k$. 
If $w_1$ is the empty word, then $r_1=a_1$ and $r_2=a_1w_2a_1w'_2$. 
By Claim~\ref{clm:triple_line}, 
the number of times that $a_i$ and $a_i^{-1}$ appear in $r_2$ is exactly $3$ for each $i\in\{2,\ldots,k+1\}$, 
and that in $w_2w'_2$ is also $3$. 
Then $w_2w'_2$ can not be reduced to the empty word, which is a contradiction. 
Supposing $w_1$ is not the empty word, 
we can check that $r_1$ and $r_2$ has a common subword, a contradiction to Claim~\ref{clm:no_common_subwords}. 

The case (iii) also leads to a contradiction in the same way as in (ii). 

In all cases, we derive a contradiction. 
Hence $n\geq k+1$, which completes the proof of Theorem~\ref{thm:mainthm}. 
\end{proof}

\begin{corollary}
\label{cor:surj}
The special shadow-complexity of closed $4$-manifolds is a surjection to $\Z_{\geq0}\setminus\{1\}$.
\end{corollary}
\begin{proof}
As seen in Theorem~\ref{thm:sc^sp_at_most_one}, 
there exist closed $4$-manifolds with special shadow-complexity $0$ 
but no closed $4$-manifolds with special shadow-complexity $1$. 
We have found a closed $4$-manifold with special shadow-complexity $k+1$ 
for each positive integer $k$. 
\end{proof}

\begin{corollary}
\label{cor:lower_bound}
For any closed $4$-manifold $W$, 
$\mathrm{rank}(\pi_1(W))\leq\spshco(W)$. 
Moreover, if $\pi_1(W)$ is the free group, then $\mathrm{rank}(\pi_1(W))+1\leq\spshco(W)$. 
\end{corollary}
\begin{proof}
Let $W$ be a closed $4$-manifold and  
$X$ a shadow of $W$ with $\spshco(W)=c(X)=n$. 
In the same way as in the proof of Theorem~\ref{thm:mainthm}, 
we can always obtain a presentation of $\pi_1(W)$ with $n$ generators. 
Therefore, we have $\mathrm{rank}(\pi_1(W))\leq\spshco(W)$. 

In order to show $\shco(\#_k(S^1\times S^3))\geq k+1$ in the proof of Theorem~\ref{thm:mainthm}, 
we only use that the group $\pi_1(\#_k(S^1\times S^3))$ is free. 
Hence, the latter statement also holds. 
\end{proof}



\begin{thebibliography}{99}
\bibitem{CM17}
A. Carrega and B. Martelli, 
{\itshape Shadows, ribbon surfaces, and quantum invariants}, 
Quantum Topol. {\bf 8} (2017), 249--294. 

\bibitem{Cos06}
F. Costantino,
{\it Stein domains and branched shadows of $4$-manifolds},
Geom. Dedicata {\bf 121} (2006), 89--111.

\bibitem{Cos05}
F. Costantino, 
{\it Shadows and branched shadows of $3$- and $4$-manifolds}.
Scuola Normale Superiore, Edizioni della Normale, Pisa, Italy, 2005. 

\bibitem{Cos06b}
F. Costantino,
{\it Complexity of $4$-manifolds}, 
Experiment. Math. {\bf 15} (2006), no. 2, 237--249.

\bibitem{Cos08} 
F. Costantino, 
{\itshape Branched shadows and complex structures on $4$-manifolds}, 
J. Knot Theory Ramifications {\bf 17} (2008), 1429--1454. 

\bibitem{CT08}
F. Costantino and D.~Thurston,
{\it $3$-manifolds efficiently bound $4$-manifolds},
J. Topol. {\bf 1} (2008), no. 3, 703--745. 

\bibitem{IK14}
M. Ishikawa and Y. Koda, 
{\itshape Stable maps and branched shadows of $3$-manifolds}, 
Math. Ann. {\bf 367} (2017), 1819--1863 

\bibitem{KMN18}
Y. Koda, B. Martelli and H. Naoe, 
{\it Four-manifolds with shadow-complexity one}, 
Ann. Fac. Sci. Toulouse. : Math\'ematiques, Serie 6, {\bf 31} (2022) no. 4, pp. 1111--1212.

\bibitem{KN20}
Y. Koda and H. Naoe, 
{\it Shadows of acyclic $4$-manifolds with sphere boundary},
Algebr. Geom. Topol. {\bf 20} (2020), no. 7, 3707--3731

\bibitem{LP72}
F. Laudenbach, V. Po\'{e}naru, 
{\it A note on $4$-dimensional handlebodies}, 
Bull. Soc. Math. France {\bf 100} (1972), 337--344.

\bibitem{Mar05} 
B. Martelli, 
{\it Links, two-handles, and four-manifolds}, 
Int. Math. Res. Not. IMRN  {\bf 2005},  no. 58, 3595--3623. 

\bibitem{Mar11} 
B. Martelli, 
{\it Four-manifolds with shadow-complexity zero}, 
Int. Math. Res. Not. IMRN  {\bf 2011},  no. 6, 1268--1351. 

\bibitem{Tur94}
V.G. Turaev, 
{\it Quantum invariants of knots and $3$-manifolds}, 
De Gruyter Studies in Mathematics, vol 18, Walter de Gruyter \& Co., Berlin, 1994.

\end{thebibliography}
\end{document}